\documentclass[12pt]{siamart0516}
\usepackage{amsfonts}
\usepackage{accents}
\usepackage{geometry}
\usepackage{fleqn}

\newcommand{\vertiii}[1]{{\left\vert\kern-0.25ex\left\vert\kern-0.25ex\left\vert #1\right\vert\kern-0.25ex\right\vert\kern-0.25ex\right\vert}}

\begin{document}
\title{On pressure estimates for the Navier-Stokes equations}

\author{J. A. Fiordilino\thanks{Supported by the DoD SMART Scholarship.  Partially supported by NSF grants CBET 1609120 and DMS 1522267.}}
\maketitle
\begin{abstract}
	This paper presents a simple, general technique to prove finite element method (FEM) pressure stability and convergence.  Typically, pressure estimates are ignored in the literature.  However, full reliability of a numerical method is not established unless the solution pair $(u,p)$ is treated.  The simplicity of the proposed technique puts pressure estimates within reach of many existing and future numerical methods and lends itself to the numerical analyst's toolbox.
\end{abstract}
\section{Introduction}
In this paper, we illustrate a simple, general technique for proving stability and convergence of FEM pressure approximations for the incompressible Navier-Stokes equations (NSE).  Pressure estimates are achievable but typically ignored due to the difficulty in obtaining them; see, e.g., Heywood and Rannacher \cite{Heywood}, Ingram \cite{Ingram}, and Labovschii \cite{Labovschii}.  However, a complete numerical analysis establishes full reliability of a numerical method.  Therefore, a complete treatment should be given to the velocity-pressure solution pair $(u,p)$ whenever possible.  Consequently, scientists and engineers will be more confident in an algorithm's capabilities and limitations.

The proposed method utilizes the discrete inf-sup condition, (\ref{infsup}), and equivalence of the dual norms $\| \cdot \|_{X_{h}^\ast}$ and $\| \cdot \|_{V_{h}^\ast}$, Lemma \ref{glemma} in Section \ref{sectfe}.  Regarding the latter, this equivalence was established by Galvin in \cite{Galvin}.  Galvin used this result to prove pressure stability in $L^{2}(0,t^{\ast};L^{2}(\Omega))$ by bounding the discrete time derivative of the velocity approximation.  Later, Zhang, Hou, and Zhao \cite{Zhang} followed the stability proof to produce an error estimate for the pressure in $L^{1}(0,t^{\ast};L^{2}(\Omega))$.  However, as is common in intricate, technical analyses, their result was correct but a gap was present in their analysis.

Herein, we illustrate the technique for linearly implicit Backward Euler (LIBE), (\ref{libe}).  It will be easy to see that our analysis can be extended to many timestepping methods including linearly implicit BDF-k and Crank-Nicolson variants.  Alternatively, for penalty and artificial compressibility methods, where the velocity approximation is not weakly divergence free, weaker forms of stability and convergence can be proven in the fully discrete setting \cite{Fiordilino2}; see, e.g., Shen \cite{Shen} for semi-discrete analyses.

Recall the NSE.  Let $\Omega \subset \mathbb{R}^{d}$ (d = 2,3) be a convex polyhedral domain with piecewise smooth boundary $\partial \Omega$.  Given the fluid viscosity $\nu$, $u(x,0) = u^{0}(x)$, and the body force $f(x,t)$, the velocity $u(x,t):\Omega \times (0,t^{\ast}] \rightarrow \mathbb{R}^{d}$ and pressure $p(x,t):\Omega \times (0,t^{\ast}] \rightarrow \mathbb{R}$ satisfy
\begin{align}
u_{t} + u \cdot \nabla u - \nu \Delta u + \nabla p = f \; \; in \; \Omega, \notag
\\ \nabla \cdot u = 0 \; \; in \; \Omega, \label{s1}
\\ u = 0 \; \; on \; \partial \Omega. \notag
\end{align}

In Section \ref{sectmath}, we collect necessary mathematical tools.  In Section \ref{sectnum}, we present the fully discrete numerical scheme, (\ref{libe}), that will be analyzed.  A complete stability and error analysis are presented in Sections \ref{sectstab} and \ref{secterror}.  In particular, the velocity and pressure approximations are proven to be unconditionally, nonlinearly, energy stable in Theorem \ref{t1} and Corollary \ref{c1}.  First-order, optimal convergence is proven in Theorem \ref{t3} and Corollary \ref{c3}.  Lastly, we state conclusions in Section \ref{sectconc}.
%%%%%%%%%%%%%%%%%%%%%%%%%%%%%%%%%%%%%%
%%%%%%%%%%%%%%%%%%%%%%%%%%%%%%%%%%%%%%
\section{Mathematical Preliminaries}\label{sectmath}
The $L^{2} (\Omega)$ inner product is $(\cdot , \cdot)$ and the induced norm is $\| \cdot \|$.  The $L^{\infty}(\Omega)$ norm is denoted $\| \cdot \|_{\infty}$.  Further, $H^{s}(\Omega)$ denotes the Hilbert spaces of $L^{2}(\Omega)$ functions with distributional derivatives of order $s\geq 0$ in  $L^{2}(\Omega)$.  The corresponding norms and seminorms are $\| \cdot \|_{s}$ and $\vert \cdot \vert_{s}$; note that $\| \cdot \|_{0} = \vert \cdot \vert_{0} = \| \cdot \|$.  Define the Hilbert spaces,
\begin{align*}
X &:= H^{1}_{0}(\Omega)^{d} = \{ v \in H^{1}(\Omega)^d : v = 0 \; on \; \partial \Omega \}, \;
Q := L^{2}_{0}(\Omega) = \{ q \in L^{2}(\Omega) : (1,q) = 0 \}, \\
V &:= \{ v \in X : (q,\nabla \cdot v) = 0 \; \forall \; q \in Q \}.
\end{align*}
For functions $v \in X$, the Poincar\'{e}-Friedrichs inequality holds,
\begin{align*}
\|v\| \leq C_{p} \|\nabla v\|.
\end{align*}
The explicitly skew-symmetric trilinear form is denoted:
\begin{align*}
b(u,v,w) &= \frac{1}{2} (u \cdot \nabla v, w) - \frac{1}{2} (u \cdot \nabla w, v) \; \; \; \forall u,v,w \in X.
\end{align*}
\noindent It enjoys the following properties.
\begin{lemma} \label{blemma}
There exists $C_{1}$ and $C_{2}$ such that for all u,v,w $\in$ X, $b(u,v,w)$ satisfies
\begin{align*}
b(u,v,w) &= (u \cdot \nabla v, w) + \frac{1}{2} ((\nabla \cdot u)v, w), \\
b(u,v,w) &\leq C_{1} \| \nabla u \| \| \nabla v \| \| \nabla w \|, \\
b(u,v,w) &\leq C_{2}  \sqrt{\| u \| \|\nabla u\|} \| \nabla v \| \| \nabla w \|.
\end{align*}
\begin{proof}
See Lemma 2.1 on p. 12 of \cite{Temam}.
\end{proof}
\end{lemma}
The discrete time analysis will utilize the following norms, for $-1 \leq k < \infty$:
\begin{align*}
\vertiii{v}_{\infty,k} &:= \max_{0\leq n \leq N} \| v^{n} \|_{k}, \;
\vertiii{v}_{p,k} := \big(\Delta t \sum^{N}_{n = 0} \| v^{n} \|^{p}_{k}\big)^{1/p}.
\end{align*}
The weak formulation of system (\ref{s1}) is:
Find $u:[0,t^{\ast}] \rightarrow X$ and $p:[0,t^{\ast}] \rightarrow Q$ for a.e. $t \in (0,t^{\ast}]$ satisfying
\begin{align}
(u_{t},v) + b(u_,u,v) + \nu (\nabla u,\nabla v) - (p, \nabla \cdot v) &= (f,v) \; \; \forall v \in X, \\
(q, \nabla \cdot u) &= 0 \; \; \forall q \in Q.
\end{align}
\subsection{Finite Element Preliminaries}\label{sectfe}
Consider a quasi-uniform mesh $\Omega_{h} = \{K\}$ of $\Omega$ with maximum triangle diameter length $h$.  Let $X_{h} \subset X$ and $Q_{h} \subset Q$ be conforming finite element spaces consisting of continuous piecewise polynomials of degrees \textit{j} and \textit{l}, respectively.  Moreover, they satisfy the following approximation properties \cite{Ern}, $\forall 1 \leq j,l \leq k,m$:
\begin{align}
\inf_{v_{h} \in X_{h}} \Big\{ \| u - v_{h} \| + h\| \nabla (u - v_{h}) \| \Big\} &\leq Ch^{k+1} \lvert u \rvert_{k+1}, \label{a1}\\
\inf_{q_{h} \in Q_{h}}  \| p - q_{h} \| &\leq Ch^{m} \lvert p \rvert_{m}, \label{a2}
\end{align}
for all $u \in X \cap H^{k+1}(\Omega)^{d}$ and $p \in Q \cap H^{m}(\Omega)$.  Furthermore, we consider those spaces for which the discrete inf-sup condition is satisfied,
\begin{align} \label{infsup} 
\inf_{q_{h} \in Q_{h}} \sup_{v_{h} \in X_{h}} \frac{(q_{h}, \nabla \cdot v_{h})}{\| q_{h} \| \| \nabla v_{h} \|} \geq \alpha > 0,
\end{align}
\noindent where $\alpha$ is independent of $h$.  Examples include the MINI-element and Taylor-Hood family of elements \cite{Layton}.  The space of discretely divergence free functions is defined by 
\begin{align*}
V_{h} := \{v_{h} \in X_{h} : (q_{h}, \nabla \cdot v_{h}) = 0, \forall q_{h} \in Q_{h}\}.
\end{align*}
The discrete inf-sup condition implies that we may approximate functions in $V$ well by functions in $V_{h}$,
\begin{lemma} \label{vlemma}
	Suppose the discrete inf-sup condition (\ref{infsup}) holds, then for any $v \in V$
	\begin{equation*}
		\inf_{v_{h} \in V_{h}} \| \nabla (v - v_{h}) \| \leq C(\alpha)\inf_{v_{h} \in X_{h}} \| \nabla (v - v_{h}) \|.
	\end{equation*}
\end{lemma}
\begin{proof}
	See Chapter 2, Theorem 1.1 on p. 59 of \cite{Girault}.
\end{proof}
The spaces $X_{h}^{\ast}$ and $V_{h}^{\ast}$, dual to $X_{h}$ and $V_{h}$, are endowed with the following dual norms
\begin{align*}
\| w \|_{X_{h}^{\ast}} := \sup_{v_{h} \in X_{h}}\frac{(w,v_{h})}{\|\nabla v_{h}\|}, \; \;
\| w \|_{V_{h}^{\ast}} := \sup_{v_{h} \in V_{h}}\frac{(w,v_{h})}{\|\nabla v_{h}\|}.
\end{align*}
Interestingly, these norms are equivalent for functions in $V_{h}$.
\begin{lemma} \label{glemma}
	Let $w \in V_{h}$.  Then, there exists $C_{\ast}>0$, independent of $h$, such that
	\begin{equation*}
	C_{\ast} \| w \|_{X_{h}^{\ast}} \leq \| w \|_{V_{h}^{\ast}} \leq \| w \|_{X_{h}^{\ast}}.
	\end{equation*}
\end{lemma}
\begin{proof}
	See Lemma 1 on p. 243 of \cite{Galvin}.
\end{proof}
\textbf{Remark:}  The proof of the above lemma utilizes the $H^{1}(\Omega)^{d}$ stability of the $L^{2}(\Omega)^{d}$ orthogonal projection onto $V_{h}$.  This result can be extended to regular meshes under certain conditions on the mesh; see \cite{Bank,Bramble} and references therein.
%%%%%%%%%%%%%%%%%%%%%%%%%%%%%%%%%%%%%%%%
%%%%%%%%%%%%%%%%%%%%%%%%%%%%%%%%%%%%%%%%
\section{Numerical Scheme}\label{sectnum}
Denote the fully discrete solutions by $u^{n}_{h}$ and $p^{n}_{h}$ at time levels $t^{n} = n\Delta t$, $n = 0,1,...,N$, and $t^{\ast}=N\Delta t$.  The fully discrete approximations of (\ref{s1}) are
\\ \underline{LIBE:}
Given $u^{n}_{h} \in X_{h}$, find $(u^{n+1}_{h}, p^{n+1}_{h}) \in (X_{h},Q_{h})$ satisfying
\begin{align}\label{libe}
(\frac{u^{n+1}_{h} - u^{n}_{h}}{\Delta t},v_{h}) + b(u^{n}_{h},u^{n+1}_{h},v_{h}) + \nu (u^{n+1}_{h}, v_{h}) - (p^{n+1}_{h},\nabla \cdot v_{h}) = (f^{n+1},v_{h}) \; \; \forall v_{h} \in X_{h}, \\
(\nabla \cdot u^{n+1}_{h}, q_{h}) \; \; \forall q_{h} \in Q_{h}. \notag
\end{align}
\section{Stability Analysis}\label{sectstab}
In this section, we prove that the pressure approximation is stable in \\ $L^{1}(0,t^{\ast};L^{2}(\Omega))$ and $L^{2}(0,t^{\ast};L^{2}(\Omega))$.  We first state, without proof, Theorem \ref{t1} regarding the stability of the velocity approximation.  Stability of the pressure approximation follows from this theorem as a proven corollary.
\begin{theorem}\label{t1}
Consider the numerical scheme (\ref{libe}).  Suppose $f \in L^{2}(0,t^{\ast};H^{-1}(\Omega)^{d})$.  Then, the velocity approximation satisfies for all $N \geq 1$
\begin{equation*}
	\|u^{N}_{h}\|^{2} + \sum_{n=0}^{N-1} \|u^{n+1}_{h} - u^{n}_{h}\|^{2} + \nu \vertiii{\nabla u_{h}}^{2}_{2,0} \leq {\nu}^{-1}\vertiii{f}^{2}_{2,-1} + \|u^{0}_{h}\|^{2}.
\end{equation*}
\end{theorem}
As a consequence of the above, the following holds.
\begin{corollary}\label{c1}
Suppose Theorem \ref{t1} holds.  Then, the pressure approximation satisfies for all $N\geq 1$
\begin{align*}
\alpha \Delta t \sum_{n=0}^{N-1}\|p^{n+1}_{h}\| \leq (1+C^{-1}_{\ast})\Big[\big(C_{1}\nu^{-2}\vertiii{f}_{2,-1} + 2\sqrt{t^{\ast}}\big)\vertiii{f}_{2,-1} + \big(C_{1}\nu^{-1}\|u^{0}_{h}\| + \sqrt{\nu t^{\ast}}\big)\|u^{0}_{h}\|\Big],
\end{align*}
and
\begin{align*}
\alpha \vertiii{p_{h}}_{2,0} \leq \sqrt{3}(1+C^{-1}_{\ast})\Big[\sqrt{C^{2}_{1}C^{2}_{h}\nu^{-2} + 1}\vertiii{f}_{2,-1} + \sqrt{C^{2}_{1}C^{2}_{h}\nu^{-1} + \nu}\|u^{0}_{h}\|\Big].
\end{align*}
\end{corollary}
\begin{proof}
Consider equation (\ref{libe}).  Let $(v_{h},q_{h}) \in (V_{h},Q_{h})$ and isolate the discrete time derivative.  Then, 
\begin{align}\label{derivlibe}
(\frac{u^{n+1}_{h} - u^{n}_{h}}{\Delta t},v_{h}) = -b(u^{n}_{h},u^{n+1}_{h},v_{h}) - \nu (u^{n+1}_{h}, v_{h}) + (f^{n+1},v_{h}) \; \; \forall v_{h} \in V_{h}.
\end{align}
The terms on the right-hand side can be bounded using Lemma \ref{blemma}, the Cauchy-Schwarz inequality, and duality, respectively,
\begin{align}\label{est1libe}
-b(u^{n}_{h},u^{n+1}_{h},v_{h}) &\leq C_{1}\|\nabla u^{n}_{h}\|\|\nabla u^{n+1}_{h}\|\|\nabla v_{h}\|,
\\-\nu (\nabla u^{n+1}_{h}, \nabla v_{h}) &\leq \nu \|\nabla u^{n+1}_{h}\|\|\nabla v_{h}\|,
\\(f^{n+1},v_{h}) &\leq \|f^{n+1}\|_{-1}\|\nabla v_{h}\|.\label{est3libe}
\end{align}
Using the above estimates in equation (\ref{derivlibe}), dividing both sides by $\|\nabla v_{h}\|$, and taking the supremum over $v_{h} \in V_{h}$ yields
\begin{align}
\|\frac{u^{n+1}_{h} - u^{n}_{h}}{\Delta t}\|_{V_{h}^{\ast}} \leq \big( C_{1}\|\nabla u^{n}_{h}\| + \nu\big)\|\nabla u^{n+1}_{h}\| + \|f^{n+1}\|_{-1}.
\end{align}
Lemma \ref{glemma} then implies
\begin{align}\label{derivestlibe}
	\|\frac{u^{n+1}_{h} - u^{n}_{h}}{\Delta t}\|_{X_{h}^{\ast}} \leq C^{-1}_{\ast}\Big[\big( C_{1}\|\nabla u^{n}_{h}\| + \nu\big)\|\nabla u^{n+1}_{h}\| + \|f^{n+1}\|_{-1}\Big].
\end{align}
Now, reconsider equation (\ref{libe}) with $v_{h} \in X_{h}$.  Isolate the pressure term and use the estimates (\ref{est1libe}) - (\ref{est3libe}).  Then,
\begin{align}
(p^{n+1}_{h},\nabla \cdot v_{h}) \leq (\frac{u^{n+1}_{h} - u^{n}_{h}}{\Delta t},v_{h}) + \big( C_{1}\|\nabla u^{n}_{h}\| + \nu\big)\|\nabla u^{n+1}_{h}\|\|\nabla v_{h}\| + \| f^{n+1}\|_{-1}\|\nabla v_{h}\|.
\end{align}
Divide both sides by $\|\nabla v_{h}\|$, take the supremum over $v_{h} \in X_{h}$, and use both the discrete inf-sup condition \ref{infsup} and estimate (\ref{derivestlibe}).  Then,
\begin{align}\label{pestlibe}
\alpha \|p^{n+1}_{h}\| \leq (1+C^{-1}_{\ast})\Big[\big( C_{1}\|\nabla u^{n}_{h}\| + \nu\big)\|\nabla u^{n+1}_{h}\| + \|f^{n+1}\|_{-1}\Big].
\end{align}
Multiplying by $\Delta t$, summing from $n=0$ to $n=N-1$, and using the Cauchy-Schwarz inequality on the right-hand side terms yields
\begin{align*}
\alpha \Delta t \sum_{n=0}^{N-1} \|p^{n+1}_{h}\| \leq (1+C^{-1}_{\ast})\Big[\big( C_{1}\vertiii{\nabla u_{h}}_{2,0} + \nu\sqrt{t^{\ast}}\big)\vertiii{\nabla u_{h}}_{2,0} + \sqrt{t^{\ast}}\vertiii{f}_{2,-1}\Big].
\end{align*}
Lastly, application of Theorem \ref{t1} on the velocity terms gives
\begin{align}\label{presult1}
\alpha \Delta t \sum_{n=0}^{N-1}\|p^{n+1}_{h}\| \leq (1+C^{-1}_{\ast})\Big[\big(C_{1}\nu^{-2}\vertiii{f}_{2,-1} + 2\sqrt{t^{\ast}}\big)\vertiii{f}_{2,-1} + \big(C_{1}\nu^{-1}\|u^{0}_{h}\| + \sqrt{\nu t^{\ast}}\big)\|u^{0}_{h}\|\Big].
\end{align}
For the second result, reconsider equation (\ref{pestlibe}).  Square both sides and notice that the right-hand side is the square of the $l_{1}$ norm of the 3-vector $\chi = (C_{1}\|\nabla u^{n}_{h}\|\|\nabla u^{n+1}_{h}\|,\nu\|\nabla u^{n}_{h}\|,\| f^{n+1}\|_{-1})^{T}$.  Consequently, using $\vert \chi \vert_{l_{1}}^{2} \leq 3 \vert \chi \vert_{l_{2}}^{2}$, multiplying by $\Delta t$, and summing from $n=0$ to $n=N-1$ we have
\begin{align*}
	\alpha^{2} \vertiii{p_{h}}^{2}_{2,0} &\leq 3(1+C^{-1}_{\ast})^{2}\Big[ C^{2}_{1}\Delta t \sum_{n=0}^{N-1} \|\nabla u^{n}_{h}\|^{2}\|\nabla u^{n+1}_{h}\|^{2} + \nu^{2}\vertiii{\nabla u_{h}}^{2}_{2,0} + \vertiii{f}^{2}_{2,-1}\Big].
\end{align*}
Now, $u^{n}_{h} \in X_{h}$ for every $0\leq n \leq N-1$; that is, $\max_{0\leq n \leq N-1} \|\nabla u^{n}_{h}\| \leq C_{h} <\infty$.  Thus, using Theorem \ref{t1} yields
\begin{align}
\alpha^{2} \vertiii{p_{h}}^{2}_{2,0} &\leq 3(1+C^{-1}_{\ast})^{2} \Big[ \big(C^{2}_{1}C^{2}_{h} + \nu^{2}\big)\vertiii{\nabla u_{h}}^{2}_{2,0} + \vertiii{f}^{2}_{2,-1}\Big]
\\ &\leq 3(1+C^{-1}_{\ast})^{2}\Big[(C^{2}_{1}C^{2}_{h}\nu^{-2} + 1)\vertiii{f}^{2}_{2,-1} + \big(C^{2}_{1}C^{2}_{h}\nu^{-1} + \nu\big)\|u^{0}_{h}\|^{2}\Big]. \notag
\end{align}
The result follows after taking the square root.
\end{proof}

In the above, note that growth with respect to $t^{\ast}$ is allowed in $L^{1}(0,t^{\ast};L^{2}(\Omega))$.  Alternatively, an unknown, bounded linear growth factor $C_{h}$ is present with respect to $L^{2}(0,t^{\ast};L^{2}(\Omega))$.  Clearly, the velocity approximation is most stable; that is, stable \textit{uniformly in time} in the appropriate norms.
\section{Error Analysis}\label{secterror}
Herein, we prove optimal-order convergence of the pressure approximation in $L^{1}(0,t^{\ast};L^{2}(\Omega))$ and $L^{2}(0,t^{\ast};L^{2}(\Omega))$.  The outline of this section is identical to the last.  We first state a convergence theorem for the velocity approximation, Theorem \ref{t3}, below.  We then prove convergence for the pressure approximation as a corollary.

Denote $u^{n}$ and $p^{n}$ as the true solutions at time $t^{n} = n\Delta t$.  Assume the solutions satisfy the following regularity assumptions:
\begin{align} \label{error:regularity1}
	u &\in L^{\infty}(0,t^{\ast};X \cap H^{k+1}(\Omega)), \;  u_{t} \in L^{2}(0,t^{\ast};H^{k+1}(\Omega)),
	\\ u_{tt} &\in L^{2}(0,t^{\ast};H^{k+1}(\Omega)), \; p \in L^{2}(0,t^{\ast};Q \cap H^{m}(\Omega)).\label{error:regularity2}
\end{align}
\textbf{Remark:} These assumptions are consistent with what is seen in the literature; more commonly, increased regularity is demanded.\\
The errors for the solution variables are denoted
\begin{align*}
	e^{n}_{u} &= u^{n} - u^{n}_{h}, \; e^{n}_{p} = p^{n} - p^{n}_{h}.
\end{align*}
We begin again with a result for the velocity approximation: optimal convergence.
\begin{theorem} \label{t3}
For u and p satisfying (\ref{s1}), suppose that $u^{0}_{h} \in X_{h}$ is an approximation of $u^{0}$ to within the accuracy of the interpolant.  Then there exists constants $C, \; C_{\triangle}>0$ such that LIBE satisfies
\begin{align*}
\|e^{N}_{u}\|^{2} + \sum_{n = 0}^{N-1} \|e^{n+1}_{\hat{u}} - e^{n}_{u}\|^{2} + \nu \vertiii{\nabla e_{\hat{u}}}^{2}_{2,0} \leq C \exp(C_{\triangle}t^{\ast}) \Big(h^{2k} + \Delta t^{2} + initial \;errors\Big).
\end{align*}
\end{theorem}
As a consequence, we have:
\begin{corollary} \label{c3}
Suppose Theorem \ref{t3}  and its hypotheses hold.  Then there exists a constant $C>0$ such that LIBE satisfies
	\begin{align*}
		\alpha\Big( \sum_{n=0}^{N-1}\|e^{n+1}_{p}\| + \vertiii{e_{p}}_{2,0}\Big)\leq C \exp(\frac{C_{\triangle}t^{\ast}}{2})\Big(h^{k} + \Delta t + initial \;errors\Big).
	\end{align*}
\end{corollary}
\begin{proof}
Recall the error equation for LIBE,
\begin{align}\label{erroreq}
	(\frac{e^{n+1}_{u}-e^{n}_{u}}{\Delta t},v_{h}) + \nu(\nabla e^{n+1}_{u},\nabla v_{h}) + b(u^{n+1}-u^{n},u^{n+1},v_{h}) + b(e^{n}_{u},u^{n+1},v_{h})
	\\ + b(u^{n}_{h},e^{n+1}_{u},v_{h}) - (e^{n+1}_{p},\nabla \cdot v_{h}) = (\frac{u^{n+1}-u^{n}}{\Delta t}-u^{n+1}_{t},v_{h}) \; \; \forall v_{h} \in X_{h}.\notag
\end{align}	
Let $v_{h} \in V_{h}$ and rewrite $\Delta t^{-1}(e^{n+1}_{u}-e^{n}_{u},v_{h}) = \Delta t^{-1}(\eta^{n+1}-\eta^{n},v_{h}) - \Delta t^{-1}(\phi^{n+1}_{h}-\phi^{n}_{h},v_{h})$ where $e^{n}_{u} = (u^{n} - I_{h}(u^{n})) - (u^{n}_{h} - I_{h}(u^{n})) = \eta^{n} - \phi^{n}_{h}$.  $I_{h}(u^{n})$ is an interpolant, such as the Lagrange interpolant, of $u$ into the finite element space such that $\phi^{n}_{h} \in V_{h}$.  Moreover, note that $(e^{n+1}_{p},\nabla \cdot v_{h}) = (p^{n+1}-q^{n+1}_{h},\nabla \cdot v_{h})$, since $v_{h} \in V_{h}$, and rearrange.  Then,
\begin{align}\label{keyeq}
(\frac{\phi^{n+1}_{h}-\phi^{n}_{h}}{\Delta t},v_{h}) = (\frac{\eta^{n+1}-\eta^{n}}{\Delta t},v_{h}) + \nu(\nabla e^{n+1}_{u},\nabla v_{h}) + b(u^{n+1}-u^{n},u^{n+1},v_{h}) + b(e^{n}_{u},u^{n+1},v_{h})
\\ + b(u^{n}_{h},e^{n+1}_{u},v_{h}) - (p^{n+1}-q^{n+1}_{h},\nabla \cdot v_{h}) - (\frac{u^{n+1}-u^{n}}{\Delta t}-u^{n+1}_{t},v_{h}) \; \; \forall v_{h} \in V_{h}.\notag
\end{align}	
Using standard estimates on the right-hand side terms yields
\begin{align}\label{errorest1}
(\frac{\eta^{n+1}-\eta^{n}}{\Delta t},v_{h}) &\leq C_{p}\Delta t^{-1/2} \|\eta_{t}\|_{L^{2}(t^{n},t^{n+1};L^{2}(\Omega))} \|\nabla v_{h}\|,
\\ - \nu(\nabla e^{n+1}_{u},\nabla v_{h}) &\leq \nu \|\nabla e^{n+1}_{u}\| \|\nabla v_{h}\|,
\\ - b(u^{n+1}-u^{n},u^{n+1},v_{h}) &\leq C_{1}\Delta t^{1/2}\|\nabla u_{t}\|_{L^{2}(t^{n},t^{n+1};L^{2}(\Omega))} \|\nabla u^{n+1}\|\|\nabla v_{h}\|,
\\ - b(e^{n}_{u},u^{n+1},v_{h}) &\leq C_{1} \|\nabla e^{n}_{u}\| \|\nabla u^{n+1}\| \|\nabla v_{h}\|,
\\ - b(u^{n}_{h},e^{n+1}_{u},v_{h}) &\leq C_{1} \|\nabla u^{n+1}_{h} \| \|\nabla e^{n+1}_{u}\| \|\nabla v_{h}\|,
\\ (p^{n+1}-q^{n+1}_{h},\nabla \cdot v_{h}) &\leq \sqrt{d}\|p^{n+1}-q^{n+1}_{h}\| \|\nabla v_{h}\|,
\\ -(\frac{u^{n+1}-u^{n}}{\Delta t}-u^{n+1}_{t},v_{h}) &\leq C\Delta t^{1/2} \|u_{tt}\|_{L^{2}(t^{n},t^{n+1};L^{2}(\Omega))}\|\nabla v_{h}\|.\label{errorestf}
\end{align}
Using the above estimates in equation (\ref{keyeq}), dividing both sides by $\|\nabla v_{h}\|$, taking a supremum over $V_{h}$ and using Lemma \ref{glemma} yields 
\begin{multline}\label{keyest}
\|\frac{\phi^{n+1}_{h}-\phi^{n}_{h}}{\Delta t}\|_{X_{h}^{\ast}} \leq C^{-1}_{\ast}\Big[C_{p}\Delta t^{-1/2} \|\eta_{t}\|_{L^{2}(t^{n},t^{n+1};L^{2}(\Omega))} + \big( \nu + C_{1}\|\nabla u^{n}_{h}\|\big)\|\nabla e^{n+1}_{u}\|
\\ + C_{1}\|\nabla u^{n+1}\|\big( \|\nabla e^{n}_{u} + \Delta t^{1/2} \|\nabla u_{t}\|_{L^{2}(t^{n},t^{n+1};L^{2}(\Omega))}\|\big)
\\ + \sqrt{d}\|p^{n+1}-q^{n+1}_{h}\| + C\Delta t^{1/2} \|u_{tt}\|_{L^{2}(t^{n},t^{n+1};L^{2}(\Omega))}\Big].
\end{multline}
Reconsidering the error equation (\ref{erroreq}), splitting the pressure error term via $(e^{n+1}_{p},\nabla \cdot v_{h}) = (p^{n+1}-q^{n+1}_{h},\nabla \cdot v_{h}) - (p^{n+1}_{h} - q^{n+1}_{h},\nabla \cdot v_{h})$, and rearranging yields
\begin{multline}
(q^{n+1}_{h} - p^{n+1}_{h},\nabla \cdot v_{h}) = (\frac{\eta^{n+1}-\eta^{n}}{\Delta t},v_{h}) - (\frac{\phi^{n+1}_{h}-\phi^{n}_{h}}{\Delta t},v_{h}) + \nu(\nabla e^{n+1}_{u},\nabla v_{h}) \\ + b(u^{n+1}-u^{n},u^{n+1},v_{h}) + b(e^{n}_{u},u^{n+1},v_{h}) + b(u^{n}_{h},e^{n+1}_{u},v_{h}) - (p^{n+1}-q^{n+1}_{h},\nabla \cdot v_{h}) 
\\ - (\frac{u^{n+1}-u^{n}}{\Delta t}-u^{n+1}_{t},v_{h}) \; \; \forall v_{h} \in X_{h}.
\end{multline}	
Using the above estimates (\ref{errorest1}) - (\ref{errorestf}), dividing by $\|\nabla v_{h}\|$, taking a supremum over $v_{h} \in X_{h}$, using (\ref{keyest}), and the discrete inf-sup condition yields
\begin{align}\label{keypressureeq}
\alpha\|q^{n+1}_{h} - p^{n+1}_{h}\| \leq (1+C^{-1}_{\ast})\Big[C_{p}\Delta t^{-1/2} \|\eta_{t}\|_{L^{2}(t^{n},t^{n+1};L^{2}(\Omega))} + \big( \nu + C_{1}\|\nabla u^{n}_{h}\|\big)\|\nabla e^{n+1}_{u}\|
\\ + C_{1}\|\nabla u^{n+1}\|\big( \|\nabla e^{n}_{u}\| + \Delta t^{1/2} \|\nabla u_{t}\|_{L^{2}(t^{n},t^{n+1};L^{2}(\Omega))}\big) + \sqrt{d}\|p^{n+1}-q^{n+1}_{h}\| \notag
\\ + C\Delta t^{1/2} \|u_{tt}\|_{L^{2}(t^{n},t^{n+1};L^{2}(\Omega))}\Big]. \notag
\end{align}	
Multiplying by $\Delta t$, summing from $n=0$ to $n=N-1$, and using the Cauchy-Schwarz inequality yields
\begin{multline}
\alpha\Delta t \sum_{n=0}^{N-1}\|q^{n+1}_{h} - p^{n+1}_{h}\| \leq (1+C^{-1}_{\ast})\Big[C_{p}\Delta t^{1/2} \|\eta_{t}\|_{L^{2}(0,t^{\ast};L^{2}(\Omega))}
\\ + \big( \nu\sqrt{t^{\ast}} + C_{1}\vertiii{\nabla u_{h}}_{2,0}\big)\vertiii{\nabla e_{u}}_{2,0} + C_{1}\vertiii{\nabla u}_{2,0}\|\big( \vertiii{\nabla e_{u}}_{2,0} + \Delta t \|\nabla u_{t}\|_{L^{2}(0,t^{\ast};L^{2}(\Omega))}\big)
\\ + \sqrt{dt^{\ast}}\vertiii{p^{n+1}-q^{n+1}_{h}}_{2,0} + C\Delta t^{3/2} \|u_{tt}\|_{L^{2}(0,t^{\ast};L^{2}(\Omega))}\Big].
\end{multline}	
By the triangle inequality and $\vert \cdot \vert_{l_{1}} \leq \sqrt{N} \vert \cdot \vert_{l_{2}}$, for an N-vector,
\begin{align*}
\alpha\Delta t \sum_{n=0}^{N-1}\|e^{n+1}_{p}\| &\leq \alpha\Delta t \sum_{n=0}^{N-1}\|p^{n+1} - q^{n+1}_{h}\| + \alpha\Delta t \sum_{n=0}^{N-1}\|q^{n+1}_{h} - p^{n+1}_{h}\| 
\\ &\leq \alpha\sqrt{t^{\ast}} \vertiii{p^{n+1} - q^{n+1}_{h}}_{2,0} + \alpha\Delta t \sum_{n=0}^{N-1}\|q^{n+1}_{h} - p^{n+1}_{h}\|.
\end{align*}
Consequently,
\begin{multline}\label{add1}
\alpha\Delta t \sum_{n=0}^{N-1}\|e^{n+1}_{p}\| \leq (1+C^{-1}_{\ast})\Big[C_{p}\Delta t^{1/2} \|\eta_{t}\|_{L^{2}(0,t^{\ast};L^{2}(\Omega))} + \big( \nu\sqrt{t^{\ast}} + C_{1}\vertiii{\nabla u_{h}}_{2,0}\big)\vertiii{\nabla e_{u}}_{2,0}
\\ + C_{1}\vertiii{\nabla u}_{2,0}\|\big( \vertiii{\nabla e_{u}}_{2,0} + \Delta t \|\nabla u_{t}\|_{L^{2}(0,t^{\ast};L^{2}(\Omega))}\big) + (\alpha + \sqrt{d})\sqrt{t^{\ast}}\vertiii{p^{n+1}-q^{n+1}_{h}}_{2,0}
\\ + C\Delta t^{3/2} \|u_{tt}\|_{L^{2}(0,t^{\ast};L^{2}(\Omega))}\Big].
\end{multline}
Now, consider (\ref{keypressureeq}), square both sides, and use $\vert \cdot \vert^{2}_{l_{1}} \leq 7 \vert \cdot \vert^{2}_{l_{2}}$, for a 7-vector.  Multiplying by $\Delta t$ and summing over n from $n=0$ to $n=N-1$ yields
\begin{multline*}
\alpha^{2} \vertiii{p_{h} - q_{h}}^{2}_{2,0} \leq 7(1+C^{-1}_{\ast})^{2}\Big[C^{2}_{p}\|\eta_{t}\|^{2}_{L^{2}(0,t^{\ast};L^{2}(\Omega))} + \big(\nu^{2} + C^{2}_{1}C^{2}_{h}\big)\vertiii{\nabla e_{u}}^{2}_{2,0}
\\ + C^{2}_{1}\vertiii{\nabla u}^{2}_{\infty,0}\big( \vertiii{\nabla e_{u}}_{2,0}^{2}+ \Delta t^{2} \|\nabla u_{t}\|^{2}_{L^{2}(0,t^{\ast};L^{2}(\Omega))}\big) + d\vertiii{p-q_{h}}_{2,0}^{2} + C^{2}\Delta t^{2} \|u_{tt}\|^{2}_{L^{2}(0,t^{\ast};L^{2}(\Omega))}\Big].
\end{multline*}	
The triangle inequality gives
\begin{align}\label{add2}
\alpha^{2} \vertiii{e_{p}}^{2}_{2,0} &\leq \alpha^{2} \vertiii{p - q_{h}}^{2}_{2,0} + \alpha^{2} \vertiii{p_{h} - q_{h}}^{2}_{2,0} 
\\ &\leq 7(1+C^{-1}_{\ast})^{2}\Big[C^{2}_{p}\|\eta_{t}\|^{2}_{L^{2}(0,t^{\ast};L^{2}(\Omega))} + \big(\nu^{2} + C^{2}_{1}C^{2}_{h}\big)\vertiii{\nabla e_{u}}^{2}_{2,0} \notag
\\ &+ C^{2}_{1}\vertiii{\nabla u}^{2}_{\infty,0}\big( \vertiii{\nabla e_{u}}_{2,0}^{2}+ \Delta t^{2} \|\nabla u_{t}\|^{2}_{L^{2}(0,t^{\ast};L^{2}(\Omega))}\big) + (d + \alpha^{2})\vertiii{p-q_{h}}_{2,0}^{2} \notag
\\ &+ C^{2}\Delta t^{2} \|u_{tt}\|^{2}_{L^{2}(0,t^{\ast};L^{2}(\Omega))}\Big]. \notag
\end{align}	
Square root (\ref{add2}) and add it to (\ref{add1}).  Taking infimums over $V_{h}$ and $Q_{h}$ and applying Lemma \ref{vlemma}, the approximation properties (\ref{a1}) - (\ref{a2}), and Theorem \ref{t3} yield the result.
\end{proof}	
\section{Conclusion}\label{sectconc}
We presented a simple, general technique to prove stability and convergence of FEM pressure approximations to the NSE.  The technique required that: the discrete inf-sup condition holds and the equivalence of certain dual norms.  Such requirements are satisfied by the popular MINI-element and Taylor-Hood family of elements.  The technique was illustrated on linearly implicit Backward Euler.  Consequently, the method is seen to be applicable to many other numerical schemes.
\section*{Acknowledgements}
The author would like to thank Professor Layton for his encouragement and suggestions.
\section*{Appendix}
For completeness, we provide a proof of the $H^{1}(\Omega)^{d}$ stability of the $L^{2}(\Omega)^{d}$ orthogonal projection onto $V_{h}$ for quasi-uniform meshes.

\begin{theorem}  
Let $\Omega_{h} \subset \Omega$ be a quasi-uniform mesh and $X_{h} \subset X$ a conforming finite element space consisting of continuous piecewise polynomials of degree $j \geq 1$.  Denote $P:L^{2}(\Omega)^{d} \rightarrow V_{h}$ as the $L^{2}$-orthogonal projection onto $V_{h} \subset X_{h}$ satisfying for all $u \in L^{2}(\Omega)^{d}$:
\begin{align*}
	(Pu - u,v_{h}) = 0 \; \forall v_{h} \in V_{h}.
\end{align*}
Then, for all $u \in H^{1}(\Omega)^{d}$:
\begin{align*}
\|\nabla Pu\| \leq C\|\nabla u\|.
\end{align*}
\end{theorem}
\begin{proof}	
Fix $u \in H^{1}(\Omega)^{d}$ and let $I_{SZ}:H^{1}(\Omega)^{d} \rightarrow X_{h}$ denote the Scott-Zhang interpolant \cite{Scott}.  Consider $\|\nabla Pu\|$, add and subtract $I_{SZ}u$, and apply the triangle inequality.  Then, 
\begin{align*}
\|\nabla Pu\| \leq \|\nabla (Pu - I_{SZ}u)\| + \|\nabla I_{SZ}u\| .
\end{align*}
Since $Pu - I_{SZ}u \in X_{h}$, the inverse estimate $\Big(\|\nabla v_{h}\| \leq C_{inv}h^{-1} \|v_{h}\|\Big)$ holds.  Consequently, applying the inverse estimate to the first term and $H^{1}(\Omega)^{d}$ stability of the Scott-Zhang interpolant \cite{Ern} to the second yields
\begin{align*}
\|\nabla (Pu - I_{SZ}u)\| + \|\nabla I_{SZ}u\| \leq C_{inv}h^{-1}\|Pu - I_{SZ}u\| + C_{SZ}\|\nabla u\|.
\end{align*}
The triangle inequality and interpolation estimates give
\begin{align*}
C_{inv}h^{-1}\|Pu - I_{SZ}u\| \leq C_{inv}h^{-1}\Big(\|Pu - u\| + \|u - I_{SZ}u\|\Big) \leq C h^{-1} h\|\nabla u\| = C\|\nabla u\|.
\end{align*}
Collecting constants yields the result.
\end{proof}

\end{document}